\documentclass[11pt]{article}

\usepackage{latexsym,color,graphicx,amsmath,amssymb,amsthm}
\usepackage[symbol]{footmisc}

\def\R{{\mathbb R}}

\def\eps{{\varepsilon}}

\def\p{{\bf p}}
\def\q{{\bf q}}

\newtheorem{thm}{Theorem}[section]
\newtheorem{lem}[thm]{Lemma}
\newtheorem{clm}[thm]{Claim}
\newtheorem{cor}[thm]{Corollary}
\newtheorem{prop}[thm]{Proposition}

\theoremstyle{definition}

\begin{document}

\title{Dense graphs have rigid parts
}

\date{}
\author{
Orit E. Raz\thanks{%
Department of Mathematics, 
University of British Columbia, Vancouver, Canada.
{\sl oritraz@math.ubc.ca} }
\and 
J\'ozsef Solymosi\thanks{%
Department of Mathematics, 
University of British Columbia, Vancouver, Canada.
{\sl solymosi@math.ubc.ca} }
}

\maketitle

\begin{abstract}
While the problem of determining whether an embedding of a graph $G$ in $\R^2$ is {\it infinitesimally rigid} is well understood, specifying whether a given embedding of $G$ is {\it rigid} or not is still a hard task that usually requires ad hoc arguments.
In this paper, we show that {\it every} embedding (not necessarily generic) of a dense enough graph (concretely, a graph with at least $C_0n^{3/2}\log n$ edges, for some absolute constant $C_0>0$), 
which satisfies some very mild general position requirements
(no three vertices of $G$ are embedded to a common line), 
must have a subframework of size at least three which is rigid.
For the proof we use a connection, established in Raz~\cite{Raz}, between the notion of graph rigidity and configurations of lines in $\R^3$. This connection allows us to  use properties of line configurations established in Guth and Katz~\cite{GK2}.
In fact, our proof requires an extended version of Guth and Katz result; the extension we need is proved by J\'anos Koll\'ar in an Appendix to our paper.

We do not know whether our assumption on the number of edges being $\Omega(n^{3/2}\log n)$ is tight, and we provide a construction that shows that requiring  $\Omega(n\log n)$ edges is necessary.
\end{abstract}


\section{Introduction}
Let $G=([n],E)$ be a graph on $n$ vertices and $m$ edges, and let $\p=(p_1,\ldots,p_n)$ be an embedding of the vertices of $G$ in $\R^2$.
A pair $(G,\p)$ of a graph and an embedding is called a {\it framework}.
A pair of frameworks $(G,\p)$ and $(G,\q)$ are {\it equivalent} if for every edge $\{i,j\}\in E(G)$ we have $\|p_i-p_j\|=\|q_i-q_j\|$, where $\|\cdot\|$ stands for the standard Euclidean norm in $\R^2$. Two frameworks are {\it congruent} if there is a rigid motion of $\R^2$ that maps $p_i$ to $q_i$  for every $i$; equivalently, if $\|p_i-p_j\|=\|q_i-q_j\|$ for every pair $i,j$ (not necessarily in $E(G)$).
We say a framework $(G,\p)$ is rigid if there exists a neighborhood $B$ of $\p$ (in $(\R^2)^n$), such that, for every equivalent framework $(G,\p')$, with $\p'\in B$, we have that the two frameworks are in fact congruent. 

It is known that if for a graph $G$ there exists a rigid framework $(G,\p)$, for some embedding $\p$, then in fact every generic embedding $\p$ gives a rigid framework of $G$ (see \cite{AR1}).
In this sense one can define the notion of rigidity of an abstract graph $G$ in $\R^2$, without specifying an embedding. That is, a graph $G$ is {\it rigid} in $\R^2$ if a generic embedding $\p$ of its vertices in $\R^2$ yields a rigid framework $(G,\p)$.
A graph $G$ is {\it minimally rigid} if it is rigid and removing any of its edges results in a non-rigid graph. Graphs that are minimally rigid in $\R^2$ have a simple combinatorial characterization, described by Laman~\cite{Lam70} (and in fact earlier by Hilda~\cite{Hil}. Namely, a graph $G$ with $n$  vertices is minimally rigid in $\R^2$ if and only if $G$ has exactly $2n-3$ edges and every subgraph  of $G$ with $k$ vertices has at most $2k-3$ edges. Every rigid graph has a minimally rigid subgraph.

To see that rigidity is indeed a generic notion one defines the stricter notion of {\it infinitesimal rigidity}.
Given a graph $G$ as above, consider the map $f_G:(\R^2)^{n}\to\R^m$, given by
$$
\p\mapsto (\|p_i-p_j\|)_{\{i,j\}\in E},
$$
for some arbitrary (but fixed) ordering of the edges of $G$. 
Let $M_G$ be the Jacobian matrix of $f_G$ (which is an $m\times 2n$ matrix).
A framework $(G,\p)$ is called {\it infinitesimally rigid} if the rank of $M_G$ at $\p$ is exactly $2n-3$. It is not hard to see that the rank of $M_G$ is always at most $2n-3$. Combining this with the fact that a generic embedding $\p$ achieves the maximal rank of $M_G$, one concludes that being infinitesimally rigid is a generic property.  As it turns out (and not hard to prove), infinitesimal rigidity of $(G,\p)$ implies rigidity of $(G,\p)$, and therefore it follows that rigidity is a generic notion too. Moreover, for rigid graphs $G$, it is straightforward to describe a (measure zero) subset $X$ of $\R^{2n}$ where the rank of $M_G$
is strictly smaller than $2n-3$, and thus for such embeddings $\p$ the framework $(G,\p)$ is not infinitesimally rigid. However, $(G,\p)$ might be rigid or not.

Finding a concrete description of all embeddings $\p$ of the vertices of $G$ that yield a rigid framework is a hard task. These are known only in very concrete and ``rare'' cases.

\paragraph{Our results.}
In this paper, we show that {\it every} embedding (not necessarily generic) of a dense enough graph,  
which satisfies some very mild general position requirements, 
must have a subframework of size at least three which is rigid.
Concretely, we prove the following theorem.
\begin{thm}\label{main}
There exists an absolute constant $C_0$ such that the following holds. Let $G$ be a graph on $n$ vertices and $C_0n^{3/2}\log n$ edges. Let $\p=(p_1,\ldots,p_n)$ be an (injective) embedding of the vertices of $G$ in $\R^2$ such that no three of the vertices are embedded to a common line.
Then there exists a subset $S\subset [n]$ of size at least three, such that the framework $(G[S],P_S)$, where $P_S:=\{p_i\mid i\in S\}$, is rigid.
\end{thm}

We do not know whether the assumption that $G$ has $\Omega(n^{3/2}\log n)$ edges in Theorem~\ref{main} is necessary, and in fact we believe an analogue statement should hold for graphs with less edges. The following theorem yields a lower bound on the number of edges, namely, $\Omega(n\log n)$, needed for the conclusion in Theorem~\ref{main} to hold.
\begin{thm}\label{main2}
For every $d\ge 2$, there exists a graph $H_d$, with $n=2^d$ vertices and 
$\frac12n\log n$ edges, and an embedding $\p$ of $H_d$ in $\R^2$, such that
no three vertices of $H_d$ are embedded to a common line in $\R^2$
and every subframework of $(H_d,\p)$ of size at least three is non-rigid.
\end{thm}

The paper is organized as follows. In Section~\ref{sec:pre}, we review a connection established in \cite{Raz} between rigidity questions and certain line configurations in $\R^3$. In Section~\ref{sec:bipartite}, we establish some properties regarding embeddings of complete bipartite graphs in $\R^2$.
In Section~\ref{sec:GK}, we review results from \cite{GK2} regarding point-line incidences in $\R^3$. In Section~\ref{sec:proofmain}, we give the proof of Theorem~\ref{main}.
In Section~\ref{sec:hypercube}, we provide a construction that proves Theorem~\ref{main2}.

\section{Rigidity in the plane and line configurations in $\R^3$}\label{sec:pre}
In this section we review some known facts that we need for our analysis. We review a reduction, introduced first in Raz~\cite{Raz}, to connect the notion of graph rigidity of planar structures with line configurations in $\R^3$.
The reduction uses the so called Elekes--Sharir framework, see~\cite{ES,GK2}.
Specifically, we represent each orientation-preserving rigid motion of the plane 
(called a \emph{rotation} in \cite{ES,GK2}) as a point
$(c,\cot (\theta/2))$ in $\R^3$, where $c$ is the center of rotation, 
and $\theta$ is the (counterclockwise) angle of rotation. (Note that pure translations are mapped
in this manner to points at infinity.) Given a pair of distinct points $a,b\in\R^2$,
the locus of all rotations that map $a$ to $b$ is a line $\ell_{a,b}$ in the above
parametric 3-space, given by the parametric equation 
\begin{equation} \label{line}
\ell_{a,b} = \{ \left( u_{a,b} + tv_{a,b},\;t\right) \mid t\in\R \} ,
\end{equation}
where $u_{a,b} = \tfrac12(a+b)$ is the midpoint of $ab$, and $v_{a,b}=\tfrac12(a-b)^{\perp}$ is a vector
orthogonal to $\vec{ab}$ of length $\tfrac12\|a-b\|$, with $\vec{ab}$, $v_{a,b}$ positively oriented
(i.e., $v_{a,b}$ is obtained by turning $\vec{ab}$ counterclockwise by $\pi/2$). 

Note that every non-horizontal line $\ell$ in $\R^3$ can be written 
as $\ell_{a,b}$, for a unique (ordered) pair $a,b\in\R^2$. 
More precisely, if $\ell$ is also non-vertical, the resulting $a$ and $b$ 
are distinct. If $\ell$ is vertical, then $a$ and $b$ coincide, at the intersection of $\ell$ with the $xy$-plane, 
and $\ell$ represents all rotations of the plane about this point.

A simple yet crucial property of this transformation is that, for any pair of pairs $(a,b)$
and $(c,d)$ of points in the plane, $\|a-c\|=\|b-d\|$ if and only if $\ell_{a,b}$ and $\ell_{c,d}$
intersect, at (the point representing) the unique rotation $\tau$ that maps $a$ to $b$
and $c$ to $d$. This also includes the special case where $\ell_{a,b}$ and $\ell_{c,d}$ are
parallel, corresponding to the situation where the transformation that maps $a$ to $b$
and $c$ to $d$ is a pure translation (this is the case when $\vec{ac}$ and $\vec{bd}$ are parallel and of equal length).

Note that no pair of lines $\ell_{a,b}$, $\ell_{a,c}$ with $b\neq c$ can intersect (or be parallel), 
because such an intersection would represent a rotation that maps $a$ both to $b$ and to $c$,
which is impossible.

\begin{lem}[{\bf Raz \cite[Lemma 6.1]{Raz}}] \label{cong}
Let $L = \{\ell_{a_i,b_i} \mid a_i, b_i \in \R^2,\; i=1,\ldots,r\}$ be a collection of $r\ge 3$ (non-horizontal) lines in $\R^3$.

\noindent
(a) If all the lines of $L$ are concurrent, at some common point $\tau$,
then the sequences $A=(a_1,\ldots,a_r)$ and $B=(b_1,\ldots,b_r)$ are congruent, 
with equal orientations, and $\tau$ (corresponds to a rotation that) maps $a_i$ to $b_i$, for each $i=1,\ldots,r$.

\noindent
(b) If all the lines of $L$ are coplanar, within some common plane $h$,
then the sequences $A=(a_1,\ldots,a_r)$ and $B=(b_1,\ldots,b_r)$ are congruent, 
with opposite orientations, and $h$ defines, in a unique manner, an orientation-reversing 
rigid motion $h^*$ that maps $a_i$ to $b_i$, for each $i=1,\ldots,r$.

\noindent
(c) If all the lines of $L$ are both concurrent and coplanar, then the points
of $A$ are collinear, the points of $B$ are collinear, and $A$ and $B$ are congruent.
\end{lem}

The following corollary is now straightforward.
\begin{cor}\label{rigidlines}
Let $G$ be a graph, over $n$ vertices, and let $\p$ be an embedding of $G$ in the plane. Assume that there exists an open neighborhood $B$ of $\p$ (in $(\R^2)^n$) such that for every equivalent framework $(G,\p')$, with $\p'\in B$, the lines $\ell_i:=\ell_{p_i,p_i'}$, for $i=1,\ldots,n$, are concurrent. Then the framework $(G,\p)$ is rigid.
\end{cor}
\begin{proof}
This follows from Lemma~\ref{cong}(a) and the definition of rigidity of a framework.
\end{proof}

\section{Embeddings of complete bipartite graphs in $\R^2$}\label{sec:bipartite}
We first recall a lemma and some notation introduced in Raz~\cite{Raz-dist2}. For completeness, we give all the details here.
For $\p=(p_1,\ldots,p_{d+1}), \p'=(p'_1,\ldots,p_{d+1}')\in(\R^d)^{d+1}$, we define
$$
\Sigma_{\p,\p'}:=\{(q,q')\in\R^d\times\R^d \mid \|p_i-q\|=\|p_i'-q'\|~~i=1,\ldots,d+1\} ,
$$
and let $\sigma_{\p,\p'}$ (resp., $\sigma'_{\p,\p'}$) denote the projection of $\Sigma_{\p,\p'}$
onto the first $d$ (resp., last $d$) coordinates of $\R^d\times\R^d$.

We have the following lemma.
\begin{lem}\label{bipsurf}
Let $\p,\p'$ be in general position. Then $\sigma_{\p,\p'}$ is a quadric surface, and there exists an invertible affine transformation 
$T:\R^d\to\R^d$, such that $T(\sigma_{\p,\p'})=\sigma'_{\p,\p'}$
and $(q,q')\in\Sigma_{\p,\p'}$ if and only if $q\in\sigma_{\p,\p'}$ and $q'=T(q)$.
\end{lem}
\begin{proof}
By definition, for $(q,q')\in \Sigma_{\p,\p'}$ we have
\begin{align*}
\|p_i-q\|^2=&\|p_i'-q'\|^2,~~~i=1,\ldots,d+1,
\end{align*}
or
$$
\|p_i\|^2-2p_i\cdot q+\|q\|^2=\|p_i'\|^2-2p_i'\cdot q'+\|q'\|^2,~~~i=1,\ldots,d+1.
$$
Subtracting  the $(d+1)$th equation from each of the other equations, we get the system
\begin{align*}
\|p_i\|^2-\|p_{d+1}\|^2-2(p_i-p_{d+1})\cdot q&= \|p'_i\|^2-\|p'_{d+1}\|^2-2(p_i'-p_{d+1}')\cdot q',~~~i=1,\ldots,d\\
\|p_{d+1}\|^2-2p_{d+1}\cdot q+\|q\|^2&=\|p_{d+1}'\|^2-2p_{d+1}'\cdot q'+\|q'\|^2.
\end{align*}
The system can be rewritten as
\begin{align*}
\tfrac12u-A q&=\tfrac12 v-B q', \\
\|p_{d+1}\|^2-2p_{d+1}\cdot q+\|q\|^2&=\|p_{d+1}'\|^2-2p_{d+1}'\cdot q'+\|q'\|^2 ,
\end{align*}
where $A$ (resp., $B$) is a $d\times d$ matrix whose $i$th row equals $p_i-p_{d+1}$ 
(resp., $p'_i-p'_{d+1}$), and 
\begin{align*}
u & = \left( \|p_1\|^2-\|p_{d+1}\|^2 , \|p_2\|^2-\|p_{d+1}\|^2 ,\ldots, \|p_d\|^2-\|p_{d+1}\|^2 \right) \\
v & = \left( \|p'_1\|^2-\|p'_{d+1}\|^2 , \|p'_2\|^2-\|p'_{d+1}\|^2 ,\ldots, \|p'_d\|^2-\|p'_{d+1}\|^2 \right) 
\end{align*}
are vectors in $\R^d$. Our assumption that each of $\p,\p'$ is in general position 
implies that each of $A,B$ is invertible. Hence we have 
$$
q'=B^{-1}Aq+w,
$$
for $w = \frac12 B^{-1}(v-u) \in\R^d$. Let $T(q):=B^{-1}Aq+w$. 
So $(q,q')\in \Sigma_{\p,\p'}$ if and only if $q'=T(q)$ and 
\begin{equation}\label{projq}
\|p_{d+1}\|^2-2p_{d+1}\cdot q+\|q\|^2=\|p_{d+1}'\|^2-2p_{d+1}'\cdot T(q)+\|T(q)\|^2,
\end{equation}
where the latter constraint comes from considering the $(d+1)$st equation, using $q'=T(q)$.
We conclude that $\sigma_{\p,\p'}$ is the quadric given by  \eqref{projq}. Moreover,  $(q,q')\in \Sigma_{\p,\p'}$ if and only if  $q\in \sigma_{\p,\p'}$  and $q'=T(q)$. 
Hence, $T$ maps $\sigma_{\p,\p'}$ into $\sigma'_{\p,\p'}$. 
This completes the proof.
\end{proof}

We now apply Lemma~\ref{bipsurf} to describe the non-rigid frameworks of $K_{3,m}$ embedded in $\R^2$.
\begin{lem}\label{lem:bipartite}
Let $K_{3,m}$ denote the $3\times m$ complete bipartite graph and let $\p:[3]\to\R^2$ and $\q:[m]\to \R^2$ be an embedding of the vertices of $K_{3,m}$ in the plane.
Suppose $m\ge 5$. Then the framework $(K_{3,m},\p\cup\q)$ is rigid, unless $\p\cup \q$ embeds the vertices of the graph to a pair of two lines in $\R^2$. 
\end{lem}
\begin{proof}
By Bolker and Roth~\cite{BR80}, a framework $(K_{3,m},\p,\q)$ is infinitesimally rigid 
in $\R^2$ if and only if $\p\cup\q$ embeds the vertices of the graph to a conic section in $\R^2$. 
(In fact, we only need the property that if the embedding is not on a conic section, then the framework is rigid.)
Since infinitesimal rigidity implies rigidity, we only need to consider the case where the image of $\p\cup \q$ is a conic section. 

Assume first that the points $\p=(p_1,p_2,p_3)$ lie on a common line in $\R^2$. In this case, the conic section supporting $\p\cup \q$ is necessarily a pair of two lines. So in this case we are done. 

Assume next that $\p=(p_1,p_2,p_3)$ are not collinear, and that $\p\cup \q$ is irreducible. 
Let $B$ be a neighborhood of $\p\cup\q$ and let $(\p',\q')\in B$ be an embedding of the vertices of $K_{3,m}$ to this neighborhood. Taking $B$ sufficiently small, we may assume that also $\p'=(p_1',p_2',p_3')$ are not collinear. 

We apply Lemma~\ref{bipsurf} to the pair $(\p,\p')$.  
Then there exists an affine transformation $T:\R^2\to \R^2$, and a quadric surface $\sigma_{\p,\p'}$ such that each of $\q,\q'$ lies on a conic section in $\R^2$ (namely, the points of $\q$ lie on $\sigma_{\p,\p'}$ and the points of $\q'$ lie on $\sigma_{\p,\p'}'=T(\sigma_{\p,\p'})$, and we have $q_j'=T(q_j)$ for every $j=1,\ldots,m$.

Recall that $\p\cup \q$ also lies on a conic section. Since two distinct conic sections can share at most four points, and using $m\ge 5$, we conclude that $\sigma_{\p,\p'}$ and the conic section supporting $\p\cup\q$ have a common irreducible component. But $\p\cup \q$ is supported by an irreducible conic section, and therefore $\p\cup \q\subset \sigma_{\p,\p'}$.

By the properties of $\sigma_{\p,\p'}$ given by Lemma~\ref{bipsurf}, we must have $T(p_i)=p_i'$, for each $i=1,2,3$, since $0=\|p_i-p_i\|=\|T(p_i)-p_i'\|$. This implies that $\|p_i-p_j'\|=\|p_i'-p_j\|$, for every $i,j=1,2,3$. That is, $\p,\p'$ are congruent configurations, and $T(\p\cup \q)=\p'\cup\q'$.
We conclude that $T$ is a rigid motion of $\R^2$ and that $\p\cup\q$, $\p'\cup\q'$ are congruent.

We showed that for some neighborhood $B$ of $(\p,\q)$, and for every $(\p',\q')\in B$, if the frameworks $(K_{3,m},(\p,\q)$ and $(K_{3,m},(\p',\q')$ are equivalent, then they are also congruent. So in this case, the framework $(K_{3,m},(\p,\q)$ is rigd, by definition. This completes the proof of the lemma.
\end{proof}

\begin{cor}\label{cor:noruled}
Let $(\p,\q)$ be an embedding of some $3+m$ vertices in $\R^2$, with $m\ge 5$, $\p=(p_1,p_2,p_3)$, $\q=(q_1,\ldots,q_m)$.
Suppose that for every neighborhood $B$ of $(\p,\q)$ (in $(\R^2)^{3+m}$), there exists $(\p',\q')\in B$ such that the following holds: The lines 
$L_{\p,\p'}:=\{\ell_{p_i,pi'}\mid i=1,2,3\}$ and $L_{\q,\q'}:=\{\ell_{q_i,q_i'}\mid i=1,\ldots,m\}$ lie on a (common) doubly ruled surface $Q$ in $\R^3$. Assume further that the lines of $L_{\p,\p'}$ lie on one ruling of the surface $Q$ and the lines of $L_{\q,\q'}$ on the other ruling of $Q$.
Then the embedding $(\p,\q)$ is supported by a pair of lines in $\R^2$. 
\end{cor}
\begin{proof}
Let $(\p,\q)$ be an embedding of some $3+m$ vertices as in the statement. By assumption, for every neighborhood $B$ of $(\p,\q)$ there exists $(\p',\q')\in B$ and a doubly ruled surface $Q$, such that the lines of $L_{\p,\p'}$ lie on one ruling of $Q$, and the lines of $L_{\q,\q'}$ on the other ruling of $Q$. In particular, $\ell_{p_i,p_i'}\cap \ell_{q_j,q_j'}\neq \emptyset$, for every $i\in[3]$, $j\in[m]$.

By the definition of the lines $\ell_{p_i,p_i'}, \ell_{q_j,q_j'}$, this implies that 
$\|p_i-q_j\|=\|p_i'-p_j'\|$ for every $i\in[3]$, $j\in[m]$. In other words, regarding $(\p,\q)$ and $(\p',\q')$ as embeddings of the graph $K_{3,m}$, we see that the frameworks $(K_{3,m},(\p,\q))$ and $(K_{3,m},(\p',\q'))$ are equivalent. Note that these frameworks are not congruent, since the lines $L_{\p,\p'}\cup L_{\q,\q'}$ are neither concurrent nor coplanar. 

Since such an embedding $(\p',\q')$ exists in every neighborhood $B$ of $(\p,\q)$, we conclude that the framework $(K_{3,m},(\p,\q))$ is not rigid. By Lemma~\ref{lem:bipartite}, $(\p,\q)$ is supported by a pair of lines in $\R^2$. This completes the proof.
\end{proof}

\section{Point-line incidences in $\R^3$}\label{sec:GK}
We recall the following theorem of Guth and Katz~\cite{GK2}.
\begin{thm}[{\bf Guth and Katz~\cite[Theorem 2.10]{GK2}}]\label{gk2rich}
Let $L$ be a set of $n$ lines in $\R^3$, such that at most $\sqrt n$ lines lie in any plane or any regulus. Then the number of $2$-rich points in $L$ is at most $O(n^{3/2})$.
\end{thm}

\begin{thm}[{\bf Guth and Katz~\cite[Theorem 4.5]{GK2}}]\label{gkthm45}
Let $L$ be a set of $n$ lines in $\R^3$, such that at most $\sqrt n$ lines lie in any plane. Let $k\ge 3$. Then the number of points in $\R^3$ incident to at least $k$ lines of $L$ is at most
$
O\left(n^{3/2}k^{-2}+nk^{-1}\right).
$
\end{thm}

We need a slightly refined version of Theorem~\ref{gk2rich}.
We thank J\'anos Koll\'ar for providing us with a detailed proof of the required statement; his proof (of, in fact, a slightly stronger statement) is given in the Appendix.
\begin{thm}\label{gkrefined}
Let $L$ be a set of $n$ lines in $\R^3$, such that:\\
(i) Every plane in $\R^3$ contains at most $\lceil n^{1/2}\rceil$ lines of $L$.\\
(ii) Every regulus in $\R^3$ contains at most $2n$ pairs of intersecting lines.\\
Then the number of $2$-rich points in $L$ is at most $O(n^{3/2})$. \hfill$\qed$
\end{thm}

Combining Theorems \ref{gkthm45} and \ref{gkrefined}, we conclude:
\begin{thm}\label{gkcor}
Let $L$ be a set of $n$ lines in $\R^3$, such that:\\
(i) Every plane in $\R^3$ contains at most $\lceil n^{1/2}\rceil$ lines of $L$.\\
(ii) Every regulus in $\R^3$ contains at most $2n$ pairs of intersecting lines.\\
Let $2\le k\le n$. Then the number of points in $\R^3$ incident to at least $k$ lines of $L$ is at most
$
O\left(n^{3/2}k^{-2}+nk^{-1}\right).$
\end{thm}

\section{Proof of Theorem~\ref{main}}\label{sec:proofmain}
Consider an embedding $\p=(p_1,\ldots,p_n)$ of the vertices of $G$ in the plane, such that no three of the points are collinear.
We prove the theorem by induction on the number, $n$, of vertices in $G$.
We assume that $G$ has $C_n n^{3/2}$ edges, and later optimize $C_n$, and get $C_n=C_0\log n$, for some absolute constant $C_0$, as in the statement of the theorem.
For the induction's base cases, we take $C_3\le\cdots\le C_{n_0}$ to be large enough so that for every  $3\le k\le n_0$ we will have $C_kk^{3/2}\ge \binom{k}{2}$. This means that a graph $G$ with $k$ vertices and $C_kk^{3/2}$ edges, for $3\le k\le n_0$, is necessarily the complete graph on $k$ vertices. Since every framework of the complete graph is rigid, this proves the base case.

Assume that the statement is true for every $n'$ with 
$3\le n'<n$ and we prove it for $n$.

\paragraph{An associated line configuration in $\R^3$.} Let $\p'=(p_1',\ldots,p_n')$ be another embedding of the vertices of $G$, taken from a neighborhood $B$ of $\p$, with the property that for every edge $\{i,j\}$ of $G$, we have 
$\|p_i-p_j\|=\|p_i'-p_j'\|$. That is, we take $\p'$ such that the frameworks $(G,\p)$ and $(G,\p')$ are equivalent.
Assume further that each $p_i'$ is taken from a small neighborhood of $p_i$ so that in particular no three points of $\p'$ are collinear.
Moreover, we may assume that no triple $p_i',p_j',p_k'$ is the reflection of $p_i,p_j,p_k$. Indeed, taking the neighborhoods of the points $p_i$ sufficiently small we can ensure that the orientation (sign of the determinant of the vectors $\overrightarrow{p_ip_j},\overrightarrow{p_ip_k}$) is the same in $\p$ and in $\p'$ for every triple $i,j,k$.

For each $i=1,\ldots, n$ put $\ell_i:=\ell_{p_i,p_i'}$ and consider the set of lines 
$
L=\{\ell_1,\ldots,\ell_n\}.
$
Note that for every edge $\{i,j\}$ in $G$, the corresponding lines $\ell_i,\ell_j$ necessarily intersect.
The other direction is not true; that is, the lines $\ell_i,\ell_j$ may intersect even if $\{i,j\}$ is not an edge in $G$.

Our assumptions on $\p$ and $\p'$, combined with Lemma~\ref{cong},  
imply that no three lines of $L$ lie on a common plane.

We claim that taking the neighborhood $B$ of $\p$ to be sufficiently small, and taking $\p'\in B$, we can guarantee that no eight lines of $L$ lie on a common regulus $R$ with at least three lines on each of the rulings of $R$ (note that this means in particular that no regulus in $\R^3$ contains more than $2n$ pairs of intersecting lines).
Indeed, fix any ordered $8$-tuple $\pi=(p_{i_1},\ldots,p_{i_8})$ (a subset of the points of $\p$). Applying Corollary~\ref{cor:noruled} (with $m=5$), and using our assumption that no three points of $\p$ are collinear, we get that for some neighborhood $B_\pi$ of $\pi$, and for every $\pi'=(p_{i_1}',\ldots,p_{i_8}')\in B_\pi$, the lines 
$\{\ell_{p_{i_1},p_{i_1}'},\ldots,\ell_{p_{i_8},p_{i_8}'}\}$ do not lie on a common regulus such that $\{\ell_{p_{i_1},p_{i_1}'},\ell_{p_{i_2},p_{i_2}'},\ell_{p_{i_3},p_{i_3}'}\}$ lie on one ruling of the regulus and $\{\ell_{p_{i_1},p_{i_1}'},\ldots,\ell_{p_{i_8},p_{i_8}'}\}$ on the other ruling of the regulus. 
Repeating this for each ordered $8$-tuples of $\p$, we see that there exists a neighborhood $B$ of $\p$ such that the claim follows.

Note in addition that, by Corollary~\ref{rigidlines}, if for every choice of $\p'$, in any arbitrarily small neighborhood of $\p$, the lines of $L$ are concurrent, this means that the framework $(G,\p)$ is rigid, and we are done. We therefore assume that the lines of $L$ are not concurrent.

\paragraph{No dense subgraphs of $G$.} Note that, by our induction hypothesis, if $G$ contains a subgraph with $3\le n'<n$ vertices and $C_{n'}(n')^{3/2}$ edges, we are done. 
Therefore we assume that every subgraph of $G$ with $3\le n'<n$ vertices has less than $C_{n'}(n')^{3/2}$ edges.

We call a point in $\R^3$ {\it $k$-rich} if it is incident to exactly $k$ lines of $L$.
Such a point is the intersection point of exactly $\binom{k}{2}$ pairs of lines, but possibly only a subset of those pairs correspond to edges of $G$.
Our assumption that $G$ has no dense subgraphs implies in particular, that for every $k$-rich point, with $3\le k<n$, the number of pairs of lines meeting at that point that also form an edge in $G$ is at most $C_kk^{3/2}$.

Clearly, every $2$-rich point, is the intersection of exactly one pair of lines and hence corresponds to at most one edge of $G$. We set $C_2$ to satisfy $C_22^{3/2}\ge 1$. 

For $t=2,\ldots,\log n$, let $E_t\subset E$ be the subset of edges that meet at a $k$-rich point for $2^{t-1}\le k\le 2^{t}$. 
Clearly, we have $E=\bigcup_{t=2}^{\log n}E_t$. 
We apply Theorems~\ref{gkcor} to upper bound $\sum_{t=1}^{\log (n/d)}|E_t|$, for some parameter $d$, which we choose later. We split the sum into two separate sums, according to which additive term in the bound from Theorem~\ref{gkcor} dominates. 

\paragraph{Edges meeting at a $k$-rich point, for $2\le k \le n^{1/2}$:}
For $2\le t<\tfrac12 \log n$, we have, by Theorem~\ref{gkcor}, that
$$
|E_t|\le \frac{\rho n^{3/2}}{2^{2(t-1)}}\cdot C_{2^t}(2^{t})^{3/2}
=4\rho C_{2^t}n^{3/2}\frac{1}{2^{t/2}},
$$
where $\rho$ is some absolute constant (given implicitly in Theorem~\ref{gkcor}).
Thus
\begin{align*}
\sum_{t=2}^{\lfloor\tfrac12 \log n\rfloor}|E_t|
&\le 4\rho C_{n^{1/2}}n^{3/2}\sum_{t=2}^{\lfloor\tfrac12 \log n\rfloor}\frac{1}{2^{t/2}}\\
&\le 4\rho C_{n^{1/2}}n^{3/2}\cdot \frac{\tfrac12(1-\frac{2^{1/2}}{n^{1/4}})}{1-2^{-1/2}}\\
&\le \rho'C_{n^{1/2}}n^{3/2},
\end{align*}
for some absolute constant $\rho'$.

\paragraph{Edges meeting at a $k$-rich point, for $n^{1/2}\le k \le n/d$:}
Similarly, for $\tfrac12 \log n \le t\le \log (n/d)$, where $d>2$ is a parameter, we have
$$
|E_t|\le \frac{\rho n}{2^{t-1}}\cdot C_{2^t}(2^{t})^{3/2}=2\rho C_{2^t}n2^{t/2},
$$
for some absolute constant $\rho$.
Thus
\begin{align*}
\sum_{t=\lceil \tfrac12 \log n\rceil}^{\lfloor \log (n/d)\rfloor}|E_t|
&\le 2\rho C_{n/d}n\sum_{t=\lceil\tfrac12 \log n\rceil}^{\lfloor \log (n/d)\rfloor}2^{t/2}\\
&\le 2\rho C_{n/d}n\cdot \frac{2^{1/2}n^{1/4}}{2^{1/2}-1}\left(\left(\frac{n^{1/2}}{d}\right)^{1/2}-1\right)\\
&\le \frac{\rho''}{\sqrt d}C_{n/d}n^{3/2},
\end{align*}
for some absolute constant $\rho''$.

Combining the two inequalities above, we get 
\begin{equation}\label{poor}
\sum_{t=2}^{\lfloor \log (n/d)\rfloor}|E_t|\le B\left(C_{n^{1/2}}+\frac{C_{n/d}}{\sqrt d}\right)n^{3/2},
\end{equation}
where $B:=\max\{\rho',\rho''\}$ is an absolute constant. That is, \eqref{poor} gives an upper bound on the number of edges of $G$ that correspond to pairs of lines meeting at a $k$-rich point, with $2\le k\le n/d$.

Recall our assumption that $G$ has at least $C_nn^{3/2}$ edges (and each edge corresponds to a pair of meeting lines of $L$).
We take $C_n$ so that 
$$
C_n\ge 2B\left(C_{n^{1/2}}+\frac{C_{n/d}}{\sqrt d}\right).$$
With this choice, and in view of \eqref{poor}, we get that 
$$
\sum_{t=2}^{\lfloor \log (n/d)\rfloor}|E_t|\le \frac12 C_n n^{3/2}.
$$
We conclude that at least half of the edges of $G$ meet at a $k$-rich point, for $k> n/d$.
In particular, there exists a point which is $k$-rich, with $k> n/d$.

\paragraph{$\alpha n$-rich point.} Assume first that there exists a point which is $\alpha n$-rich, with $1/d\le \alpha\le 2/3$. Let $L_1$ denote the subset of $\alpha n$ lines going through this point. If the number of edges meeting at that point (i.e., the number of pairs of lines of $L_1$ that correspond to an edge in $G$) is at least $C_{\alpha n}(\alpha n)^{3/2}$, then we are done by induction. Consider the subset of lines $L_2:=L\setminus L_1$ that do not go through this $\alpha n$-rich point. If the number of edges induced by $L_2$  is at least $C_{(1-\alpha)n}((1-\alpha)n)^{3/2}$, we are again done by induction.
Finally, note that every line of $L_2$ intersects at most one line of $L_1$. Otherwise, we would have three coplanar lines, contradicting our assumption.
Therefore, the total number of edges we have is at most
$$
C_{\alpha n}(\alpha n)^{3/2}+C_{(1-\alpha)n}((1-\alpha)n)^{3/2} +(1-\alpha)n,
$$
which must be at least $C_nn^{3/2}$, by our assumption on the number of edges in $G$. Thus
$$
C_{\alpha n}\alpha^{3/2}+C_{(1-\alpha)n}(1-\alpha)^{3/2} +(1-\alpha)n^{-1/2}\ge C_n.
$$
Using $C_{\alpha n}, C_{(1-\alpha)n}\le C_n$ (by monotonicity of the sequence $C_n$), this implies
$$
C_{n}(\alpha^{3/2}+(1-\alpha)^{3/2}) +(1-\alpha)n^{-1/2}\ge C_n
$$
or
\begin{equation}\label{ineq:ansmall}
\frac{1-\alpha}{C_nn^{1/2}}\ge 1-\alpha^{3/2}-(1-\alpha)^{3/2}.
\end{equation}
Using $1/d\le \alpha\le 2/3$, we have
$$
\frac{1-\alpha}{C_nn^{1/2}}\le\frac{d-1}{dC_nn^{1/2}}.
$$
Combined with \eqref{ineq:ansmall}, the last inequality implies
\begin{equation}\label{ineq:ansmall2}
1-\alpha^{3/2}-(1-\alpha)^{3/2} \le 
\frac{d-1}{dC_nn^{1/2}}.
\end{equation}
Note that for every $0<\alpha<1$, the left-hand side of \eqref{ineq:ansmall2} is positive. Moreover, for every closed interval $[a,b]\subset [0,1]$, with $0<a<b<1$, the function $f(\alpha)=1-\alpha^{3/2}-(1-\alpha)^{3/2}$ attains a minimum which is a positive number. Let $\delta_0>0$ denote the minimum of $f$ over $[1/d,2/3]$.
Taking $n_0$ large enough (and recalling that $n\ge n_0$), the right-hand side of \eqref{ineq:ansmall2} can be guaranteed to be smaller than $\delta_0$ (for any positive $\delta_0$).
This yields a contradiction to \eqref{ineq:ansmall2}.

\paragraph{$k$-rich point, with $k>2n/3$.} Assume next that there exists a $k$-rich point with $k>2n/3$. Fix such a point, and denote by $m$ the number of lines not incident to this point. That is, we fix a $(n-m)$-rich point, with $m<n/3$. Note that $m\ge 1$, by our assumption that not all the lines of $L$ are concurrent. 

Similar to the analysis in the previous case above, if the number of edges meeting at the given $n-m$ rich point is at least $C_{n-m}(n-m)^3/2$, then we are done by induction. Thus, we assume this is not the case.
Note that in this case, and if $m=2$, we get that in this case the total number of edges in $G$ is at most
$$
C_{n-2}(n-2)^{3/2}+1+2(n-2),
$$
where here we used our assumption that no three lines of $L$ lie on a common plane. So we must have 
$$
C_{n-2}(n-2)^{3/2}+1+2\ge C_nn^{3/2}
$$
which implies
$$
3\ge C_n(n^{3/2}-(n-2)^{3/2}),
$$
which yields a contradiction, taking $C_n$ larger than some absolute constant. 
So we must have $m\ge 3$. 

Next, if the number of edges among the $m$ lines not incident to our $(n-m)$-rich point is at least $C_mm^{3/2}$, we are again done by induction.
Otherwise, we have that the total number of edges is at most
$$
C_{n-m}(n-m)^{3/2}+C_m m^{3/2}+m,
$$
which, on the other hand, must be at least $C_n n^{3/2}$, since this is the total number of edges in $G$, by assumption.
Using $C_m, C_{n-m}\le C_n$, this implies 
$$
C_{n}(n-m)^{3/2}+C_n m^{3/2}+m\ge C_n n^{3/2}.
$$
or
\begin{equation}\label{ineq:anlarge}
(n-m)^{3/2}+ m^{3/2}+\frac{1}{C_n}m\ge n^{3/2}.
\end{equation}

Consider the function $f(x)=(n-x)^{3/2}+x^{3/2}+\tfrac1{C_n}x$. Note that $f$ is monotone decreasing in $x$. Indeed,
$$
f'(x)=-\tfrac32(n-x)^{1/2}+\tfrac32x^{1/2}+\tfrac{1}{C_n},
$$
and we have $f'(x)<0$ if
$$
\tfrac{2}{3C_n}<(n-x)^{1/2}-x^{1/2}.
$$
The last inequality holds for instance for every $x\le n/3$. Thus $f$ is monotone decreasing in the range of $1\le x\le n/3$.
In particular, $f(1)\ge f(m)$, for $m$ in our range, and the inequality \eqref{ineq:anlarge} implies
$$
f(1)\ge n^{3/2},$$
which is a contradiction, taking $n_0$ sufficiently large.

To summarize, in at least one of the two cases analyzed above it must be possible to apply the induction hypothesis; otherwise, in each of the two cases, we get a contradiction. This completes the proof of the theorem, for any monotone increasing function $C_n$ satisfying 
$$
C_n\ge 2B\left(C_{n^{1/2}}+\frac{C_{n/d}}{\sqrt d}\right).$$
Solving the recurrence relation, one can take $C_n=C_0\log n$, for some absolute value $C_0>0$. This completes the proof of the theorem. 
\hfill$\qed$

\section{Proof of Theorem~\ref{main2}}\label{sec:hypercube}
Let $H_d$ be the graph induced by a hypercube in $\R^d$. That is, each vertex corresponds to a $d$-tuple in $\{0,1\}^d$, and a pair of vertices are connected by an edge if and only if the corresponding $d$-tuples are different by exactly one entry. So $H_d$ has $2^d$ vertices and $d2^{d-1}$ edges. 

We now describe an embedding $\p$ of the vertices of $H_d$ in $\R^2$. For this, we start with an embedding $\bar\p$ of $H$ in $\R^d$. We take 
the standard embedding of the hypercube, namely, we map a vertex with corresponding $d$-tuple $(b_1,\ldots,b_d)$, to the point $(b_1,\ldots,b_d)$ in $\R^d$. 

\begin{clm}
No three vertices of $H_d$ are embedded by $\bar\p$ to a common line in $\R^d$. 
\end{clm}
\begin{proof}
Consider two distinct $d$-tuples $(b_1,\ldots,b_d)$ and $(b_1',\ldots,b_d')$. Assume without loss of generality that $b_1\neq b_1'$. Then,  for every $ t\in\R\setminus\{0, 1\}$, we have $tb_1+(1-t) b_1'\not\in\{0,1\}$. Thus no other point on the line connecting $(b_1,\ldots,b_d)$ and $(b_1',\ldots,b_d')$ is a vertex of $H_d$.
\end{proof}

Identify a point in $\R^{2d}$ with a $2\times d$ matrix, regarded as a linear transformation from $\R^d$ to $\R^2$.
We define $\p:=T\circ\bar\p$, where $T:\R^d\to \R^2$ is a linear transformation. We choose $T\in \R^{2d}$ so that with this choice no three distinct vertices of $H_d$ are embedded by $\p$ to a common line and no six distinct vertices of $H_d$ are embedded by $\p$ to a common conic section.
To prove the existence of such $T$ we need the following two claims.
\begin{clm}\label{genericTcoll}
Let $q_1,q_2,q_3\in\R^d$ be three distinct non-collinear points.
Then there exists an algebraic subvariety $Z\subset \R^{2d}$, of codimension at least one, such that for every $T\in \R^{2d}\setminus Z$, the points $Tq_1,Tq_2, Tq_3$ are not collinear.
\end{clm}
\begin{proof}
There exists a polynomial, $P$, over $6$ variables and with rational coefficients, such that, for every $p_1,p_2,p_3\in \R^2$, $P(p_1,p_2,p_3)=0$ if and only if the points $p_1,p_2,p_3$ are collinear.
Namely, $P$ is just the determinant of the $2\times 2$ matrix with columns $p_2-p_1$ and $p_3-p_1$.
Consider the equation 
\begin{equation}\label{collineareq}
P(Tq_1,Tq_2,Tq_3)=0.
\end{equation}
Since $q_1,q_2,q_3$ are given, this is an equation in the entries of $T$, which defines a subvariety of $\R^{2d}$.

It is easy to see that \eqref{collineareq} is not identically zero. Indeed, consider a linear transformation $T$ which maps the plane spanned by the vectors $q_2-q_1,q_3-q_1$ (this is a plane through the origin) to $\R^2$ injectively. Such $T$ does not satisfy \eqref{collineareq}. 
Thus \eqref{collineareq} defines a subvariety $Z$ of $\R^{2d}$ of codimension at least one. This proves the claim.\end{proof}

For every triple $u_1,u_2,u_3$ of vertices of $H_d$, we apply Claim~\ref{genericTcoll} to the points $q_i:=\bar\p(u_i)$ for $i=1,2,3$.
Let $\cal Z$ be the family of algebraic subvariety of $\R^{2d}$ of ``bad'' choices of $T$, given by applying Claim~\ref{genericTcoll} to each triple of vertices.
Since each element of $\cal Z$ is of codimension at least one, and $\cal Z$ is finite, the union of the elements of $\cal Z$ does not cover $\R^{2d}$. Therefore, there exists a choice of $T$ that does not lie on any of the elements of $\cal Z$. 
Using such $T$ in the definition of $\p$, we get that no three distinct vertices of $H_d$ are embedded by $\p$ to a common line.

Finally, we claim that the framework $(H_d,\p)$ does not have a rigid subframework of size larger than two.
In fact, we prove the following stronger property. 

\begin{clm} Let $x,y$ be any pair of distinct vertices of $H_d$, such that $\{x,y\}$ is not an edge of $H_d$. Consider a neighborhood, $B$, of $\p$ in $\R^2$ arbitrarily small. Then there exists an embedding $\p'\in B$, such that $\p$ and $\p'$ are equivalent, but 
$\|\p(x)-\p(y)\|\neq\|\p'(x)-\p'(y)\|$.
\end{clm}
\begin{proof}
We prove the claim by induction on $d$.
The base case $d=2$ is easy to see.
Consider $d>2$. The vertices of $H_d$ can be regarded as a disjoint union of two copies $H_{d-1}^{(1)}$, $H_{d-1}^{(2)}$of $H_{d-1}$.
Note that each vertex $u\in H_{d-1}^{(1)}$ can be associated with a vertex  $u'\in H_{d-1}^{(2)}$, such that $\{u,u'\}$ is an edge in $H_d$.
Moreover, note that by the definition of our embedding $\p$, all the edges of this form (edges between a vertex of $H_{d-1}^{(1)}$ and a vertex of $H_{d-1}^{(2)}$) have the same length $\ell$.

Let $x,y$ be a pair of distinct vertices of $H_d$ such that $\{x,y\}$ is not an edge in $H_d$.
Assume first that the pair $x,y$ is in one of the copies of $H_{d-1}$, say in $H_{d-1}^{(1)}$. Let $\q:=\p_{|_{H_{d-1}^{(1)}}}$ be the embedding $\p$ of $H$, restricted the subgraph $H_{d-1}^{(1)}$.
By the induction hypothesis, for every arbitrarily small neighborhood of $\q$,  there exists an embedding $\q'$ in this neighborhood, such that $\q,\q'$ are equivalent, but $\|\q(x)-\q(y)\|\neq \|\q'(x)-\q'(y)\|$. By the symmetry of $H_{d-1}^{(1)}$ and $H_{d-1}^{(2)}$ it is easy to see that this can be extended to an embedding $\p'$ of $H_d$ which is congruent to $\p$.
This proves the claim in this case.

Assume next that, say, $x\in H^{(1)}_{d-1}$, $y\in H^{(2)}_{d-1}$, and recall that $\{x,y\}$ is not an edge in $H_d$.
Consider a neighborhood of $\p$, arbitrarily small.
For each vertex $u\in H_{d-1}^{(1)}$, take a rotation $r_u$ of the plane centered at $u$, with angle of rotation $\eps$. We apply this rotation only to the (unique) vertex $u'\in H_{d-1}^{(2)}$ with the property that $\{u,u'\}$ is an edge in $H_d$. This induces a new embedding $\p'$ of $H_d$. 
Clearly, taking $\eps>0$ sufficiently small, $\p'$ is in the given neighborhood of $\p$. Moreover, since $\p'$ applied to the vertices of $H_{d-1}^{(2)}$ is a translation of $\p'$ applied to $H_{d-1}^{(1)}$, it is clear that by construction that $\p$ and $\p'$ are equivalent.
Finally, we claim that for $\eps$ sufficiently small, we have $\|\p(x)-\p(y)\|\neq \|\p'(x)-\p'(y)\|$.
To see this it is sufficient to restrict our attention  to the vertices $x,y'\in H_{d-1}^{(1)}$ and $x',y\in H_{d-1}^{(2)}$, where $\{x,x'\}$ and $\{y',y\}$ are edges in $H_d$.
Note that since $\{x,y\}$ is not an edge, $x,x',y,y'$ are distinct.
Also, by construction, $\|\p(x)-\p(y')\|= \|\p'(x')-\p'(y)\|$ and $\|\p(x)-\p(x')\|= \|\p'(y')-\p'(y)\|$. It is now easy to see, again by the construction of $\p'$ that 
$\|\p(x)-\p(y)\|\neq \|\p'(x)-\p'(y)\|$, as claimed.
\end{proof}

\vspace{1cm}
\noindent {\bf Acknowledgements} 
The work of the second author has received funding from the European Research Council (ERC) under the European Union’s Horizon 2020 research and
innovation programme (grant agreement No 741420, 617747, 648017). His research is also supported by NSERC and OTKA (K 119528) grants.
The authors also thank Omer Angel and Ching Wong for several useful comments regarding the paper.

\vspace{1cm}
\center{\bf\Large Appendix for ``Dense graphs have rigid parts''}
\center{\bf\large by J\'anos Koll\'ar\footnote[1]{Department of Mathematics, 
Princeton University, 
{\sl kollar@math.princeton.edu}}}
\vspace{1cm}

\noindent Let ${\mathcal L}$ be a set of $m$ distinct lines in ${\mathbb C}^3$.
A weighted number of their intersection points is
$$
I({\mathcal L}):=\sum_{p\in {\mathbb C}^3}  \bigl(r(p)-1\bigr),
$$
where $r(p)$ denotes the number of lines
passing through a point $p$. Our aim is to outline the proof of the
following variant of \cite[Theorem 6]{Kollar}.
The difference is that, unlike in  \cite[Theorem 6]{Kollar}, we allow more 
than
$2c\sqrt{m}$ lines on a regulus (that is, a smooth quadric surface), but 
we restrict the number of intersections between them.

\begin{prop} \label{thm.6.variant}
Let ${\mathcal L}$ be a set of $m$ distinct lines in ${\mathbb C}^3$.
Let $c$ be a constant such that
every   plane  contains at most  $c\sqrt{m}$  of the lines and, for every 
regulus, the lines on it have at most $c^2m$ intersection points with each 
other.
Then
$$
I({\mathcal L})\leq \bigl(29.1+\tfrac{c}2\bigr)\cdot m^{3/2}.
$$
\end{prop}

\begin{proof}Following the method of \cite{GK2}, there is an algebraic 
surface
$S$ of degree $\leq \sqrt{6m}-2$ that contains all the lines in 
${\mathcal L}$.
We decompose $S$ into its irreducible components $S=\cup_j S_j$.

Now we follow the count as in \cite[Paragraph 24]{Kollar}.
The bound for external intersections  (when a line not on $S_j$ meets a 
line on $S_j$) is the same as in \cite[Paragraph 18]{Kollar}. The remaining 
internal   intersections  (when a line  on $S_j$ meets a line on the same 
$S_j$) is done one surface at a time.
The only change is with the count on a regulus, which is done in 
\cite[Paragraph 19]{Kollar}.

Thus let $Q_j$ be a regulus that contains $n_j$ lines. If $n_j\leq 
2c\sqrt{m}$ then we use the formula on the bottom of p. 38:
$I({\mathcal L}_j)\leq \frac{c}{2}n_j\sqrt{m}$.
If $n_j\geq 2c\sqrt{m}$ then  we use that, by assumption
$$
I({\mathcal L}_j)\leq c^2m =2c\sqrt{m} \tfrac{c}{2}\sqrt{m}\leq 
n_j\tfrac{c}{2}\sqrt{m}.
$$
So
$I({\mathcal L}_j)\leq \tfrac{c}{2}n_j\sqrt{m}$
always holds for every  regulus and
this is the only information about lines on a regulus that the proof in
\cite[Paragraph 24]{Kollar} uses.  The rest of the proof is unchanged. 
\end{proof}


\begin{thebibliography}{}

\bibitem{AR1}
L. Asimow and B. Roth,
The rigidity of graphs,
{\it Trans. Amer. Math. Soc.},
245 (1978), 279--289.
%


\bibitem{BR80}
E. D. Bolker and B. Roth,
When is a bipartite graph a rigid framework?,
{\it Pacific J. Math.}, 90 (1980), 27--44.

\bibitem{ES}
Gy. Elekes and M. Sharir, 
Incidences in three dimensions and distinct distances in the plane, 
{\it Combinat. Probab. Comput.}, 20 (2011), 571--608. Also in {\tt arXiv:1005.0982}.

\bibitem{GK2}
L. Guth and N. H. Katz, 
On the Erd\H{o}s distinct distances problem in the plane, 
\emph{Annals Math.} 18 (2015), 155--190.

\bibitem{Hil}
 P.-G., Hilda (1927), \"Uber die Gliederung ebener Fachwerke, ZAMM-{\it Journal of Applied Mathematics and Mechanics/Zeitschrift f\"ur Angewandte Mathematik und Mechanik}, 7.1 (1927), 58--72.

\bibitem{Kollar}
J. Koll\'ar,
Szemer\'edi-Trotter-type theorems in dimension 3,
{\it Adv. Math.}
271 (2015), 30--61.


\bibitem{Lam70}
G. Laman,
On graphs and rigidity of plane skeletal structures, 
{\it J. Engrg. Math.} 4 (1970), 333--338.


\bibitem{Raz}
O. E. Raz,
Configurations of lines in space and combinatorial rigidity,
{\it Discrete Comput. Geom.} (special issue), 58 (2017), 986--1009.

\bibitem{Raz-dist2}
O. E. Raz,
Distinct distances for points lying on curves in $\R^d$---the bipartite case,
{\it manuscript}.



\end{thebibliography}
\end{document}